\documentclass{bcp}
\usepackage{amsmath,amssymb,amsfonts,amsthm,graphics}

\newcommand{\QED}{$\sq$\vskip.3cm}
\newcommand{\C}{\mathbb{C}}
\newcommand{\R}{\mathbb{R}}
\newcommand{\N}{\mathbb{N}}
\newcommand{\Z}{\mathbb{Z}}
\newcommand{\OD}{\mathbb{D}}
\newcommand{\D}{\overline{\mathbb{D}}}

\def\rd{{\rm d}}
\newtheorem{theorem}{Theorem}[section]
\newtheorem{corollary}[theorem]{Corollary}
\newtheorem{lemma}[theorem]{Lemma}
\newtheorem{proposition}[theorem]{Proposition}

\theoremstyle{definition}

\newtheorem{example}[theorem]{Example}

\usepackage{latexsym}

\usepackage{hyperref}

\def\sskip{\vskip 0.3cm}

\begin{document}

\mathclass{Primary 46J05; Secondary 46J45, 46J10.}
\abbrevauthors{J. F. Feinstein and H. Kamowitz}
\abbrevtitle{Endomorphisms of semiprime Banach algebras}

\title{Quasicompact endomorphisms of\\commutative semiprime Banach algebras}

\author{Joel F. Feinstein}
\address{School of Mathematical Sciences, University of Nottingham\\University Park, Nottingham~NG7~2RD, UK\\
E-mail: Joel.Feinstein@nottingham.ac.uk}

\author{Herbert Kamowitz}
\address{Department of Mathematics, University of Massachusetts at Boston\\
100 Morrissey Boulevard, Boston, MA 02125-3393, USA\\
E-mail: hkamo@cs.umb.edu}

\maketitlebcp

\begin{abstract}
This paper is a continuation of our study of  compact,
power compact,  Riesz, and  quasicompact endomorphisms of  commutative
Banach algebras. Previously it has been shown that if $B$ is a
unital commutative
semisimple Banach algebra with connected character space, and $T$ is a
unital endomorphism of $B$, then $T$ is quasicompact if and only if the
operators $T^n$ converge in operator norm to a rank-one unital endomorphism
of $B$.

In this note the discussion is extended in two ways: we discuss endomorphisms of commutative Banach algebras which are semiprime and not necessarily semisimple; we also discuss commutative Banach algebras with character spaces which are not necessarily connected.

In previous papers we have given examples of commutative semisimple Banach
algebras $B$ and
endomorphisms $T$ of $B$ showing that $T$ may be quasicompact but not Riesz, $T$ may be Riesz
but not power compact, and $T$ may be
power compact but not compact. In this note we give examples of commutative, semiprime Banach
algebras, some radical and some semisimple, for which every quasicompact endomorphism is actually
compact.

\end{abstract}

\section{Introduction}
Let $A$ be a commutative, complex Banach algebra. We denote by $\Phi_A$ the character space of $A$, and, for $a \in A$, we denote by $\hat a$ the Gelfand transform of $a$. As in \cite{dales}, if $A$ has no identity element then we denote by $A^\#$ the usual Banach algebra obtained by adjoining an identity to $A$ ; otherwise we define $A^\#=A$. We denote the open unit disc by $\OD$ and the closed unit disc by $\D$. We denote the standard disc algebra (regarded as a Banach space) by $A(\D)$.

In previous papers the authors \cite{FK3,FK4,FK5}
and others \cite{GGL1,GGL,Kl} have studied
endomorphisms of commutative semisimple Banach algebras and have obtained several
general theorems, and also a variety of results pertaining to specific classes of algebras.
In this note we extend this discussion to
endomorphisms of commutative Banach algebras which are semiprime and not necessarily semisimple.

We recall that a complex algebra $B$ is {\it semiprime} if
$J=\{0\}$ is the only ideal in $B$ such that the product of every pair of elements in $J$ is $0$.
It is standard that a commutative, complex algebra $B$ is semiprime if and only if $B$ has no
non-zero nilpotent elements (see, for example, \cite{Dix} or \cite[pp.77-78]{dales}).
Certainly semisimple algebras are semiprime.

Examples of commutative semiprime Banach  algebras which are not semisimple include certain Banach algebras of formal power series, as discussed in \cite{Grab}.
In particular, $A(\D)$ and $H^p(\OD)$ for  $p \in [1,\infty)$ are commutative
radical semiprime Banach algebras with respect to \emph{convolution}
multiplication defined by
\[(f \ast g) (z)=\int_{\gamma_z} f(z-w)g(w)\,{\rm d}w\,,\]
where the path $\gamma_z$ is a straight line joining $0$ to $z.$
Other examples of commutative radical semiprime Banach algebras include
$\ell^p(\omega)$ for $p \in [1,\infty)$ and radical weights $\omega$.\footnote{A real valued function $\omega$ on $\Z^+$ is a \emph{weight}
if $\omega(n) > 0$ for all $n \in \Z^+$ and, for all $m$ and $n$ in $\Z^+$,
we have $\omega(m + n) \leq
  \omega(m)\omega(n)$. The weight is \emph{radical} if, in
  addition, $\lim_{n \rightarrow \infty}\omega(n)^{1/n} = 0.$}

A linear map $T$ from a commutative Banach algebra $A$ to itself is an \emph{endomorphism} if
$T$ preserves multiplication. If the algebra $A$ is unital,
then an endomorphism $T$ of $A$ is
said to be \emph{unital} if $T$ maps the identity to itself.
In this case, $\phi := T^*|_{\Phi_A}$ is a selfmap of $\Phi_A$; we shall call $\phi$ the selfmap of $\Phi_A$ associated with $T$. Note that then, for all $a \in A$, we have
\[
\widehat{Ta} = \hat a \circ \phi\,.
\]
In particular, if $A$ is semisimple, then we may recover the endomorphism from the associated selfmap $\phi$. If $A$ is not semisimple, then $\phi$ may give little information about the endomorphism $T$. Even in the latter case, however, the existence or otherwise of fixed points of $\phi$ is relevant to our study of endomorphisms.

For commutative semisimple Banach
algebras, endomorphisms are automatically continuous. However, in the case of commutative semiprime algebras, this need not be the case, at least if we assume the continuum hypothesis (CH). Indeed, let $\omega$ be a radical weight on $\R^+$, and set $A=L^1(\R^+,\omega)$. Assuming CH, it follows from \cite[Theorem 5.7.31]{dales} and the comments following that theorem, that there is then a discontinuous, injective, unital endomorphism of the integral domain $A^\#$. We shall consider only bounded
endomorphisms in this note.

Let $E$ be an infinite dimensional Banach
space, let $\mathcal{L}(E)$ be the Banach algebra of bounded linear operators
on $E$, and let $\mathcal{K} (E)$ be the set of compact linear operators on $E$. Then $\mathcal{K} (E)$ is a closed ideal in $\mathcal{L}(E)$. The quotient algebra $\mathcal{L}(E)/\mathcal{K} (E)$ is called the \emph{Calkin algebra}.
Now let $T$ be a bounded linear operator on $E$.
The \emph{essential spectral radius} of $T$, $\rho_e(T)$, is the spectral radius of $T + \mathcal{K} (E)$ in the Calkin algebra.

We shall discuss operators $T$ such that $\rho_e(T) < 1.$ (This holds if and only if there is a natural number $n$ such that the distance from $T^n$ to $\mathcal{K}(E)$ is strictly less than $1$.)
Following Heuser \cite{Heu}
such an operator $T$ is called \emph{quasicompact}.
If $\rho_e(T)=0$ the operator
$T$ is a \emph{Riesz} operator. Quasicompactness is clearly weaker than Riesz,
which in turn is weaker than the condition that an operator be
power compact.

In \cite{FK5}, the authors investigated quasicompact endomorphisms of
commutative semisimple Banach algebras.
One of the main results of that paper was the following.

\begin{proposition}
\label{ssqc}
Let $B$ be a unital commutative semisimple Banach
algebra with connected character space, let $T$ be a unital
endomorphism of $B$, and let $\phi$ be the associated selfmap of $\Phi_B$.
Then $T$ is quasicompact if and only if
the operators $T^n$ converge in operator norm to a rank-one unital
endomorphism of $B$; in this case $\phi$ has a unique fixed point $x_0 \in \Phi_A$, and
the rank-one endomorphism above must be the endomorphism $b \mapsto \hat b(x_0) 1$.\QED
\end{proposition}

In Section 2, we indicate that Proposition \ref{ssqc}
is valid for bounded unital
endomorphisms of commutative semiprime Banach algebras whose character
space is connected.
In Section 3, we consider the case where the character space need not be
connected. Using a fairly standard technique involving orthogonal
idempotents, we will prove the following result, which is a main result of this
note.

\begin{theorem}
\label{general-spqc}
Let $B$ be a unital commutative semiprime Banach
algebra, and let $T$ be a bounded unital
endomorphism of $B$. Then $T$ is quasicompact if and only if
there is a natural number $n$ such that the operators $(T^{kn})_{k=1}^\infty$
converge in operator norm to a finite-rank unital endomorphism of $B$.
\end{theorem}

This result extends earlier results of the authors \cite{FK3,FK4,FK5} for commutative semisimple
Banach algebras, and results for uniform algebras of Klein \cite{Kl} and Gamelin, Galindo and Lindstr\"{o}m
\cite{GGL1,GGL}.

Section 4 contains some results about commutative radical semiprime
Banach algebras, while Section 5 presents two examples of
commutative semisimple Banach algebras where each quasicompact
endomorphism is compact.

\section{Bounded endomorphisms of semiprime Banach algebras with connected character space}

   In order to extend the results from \cite{FK5}, we begin by examining the
properties of semisimplicity which were used in the proof of Lemma
1.1 of \cite{FK5}, and observing that they are more generally true. Specifically, we
note the following.

\begin{itemize}

\item
Let $B$ be a unital commutative Banach algebra. Then $\Phi_B$ is connected if and only if  the only idempotent elements in $B$ are $0$ and $1$. This is an immediate consequence of the Shilov Idempotent Theorem \cite[Theorem 2.4.33]{dales}.
\item
Let $B$ be a commutative unital semiprime Banach algebra, and let $T$ be a
unital endomorphism of $B$. Then, since $B$ has no non-zero nilpotent
elements, the set of eigenvalues of $T$ is closed under taking powers.

\item
Let $A$ be a finite-dimensional commutative semiprime Banach algebra. Since the radical of a finite dimensional algebra is
nilpotent \cite[Theorem 1.5.6(iv)]{dales}, it follows that $A$
is, in fact, semisimple. Thus $A$ is isomorphic to the
finite-dimensional commutative C*-algebra $\C^m$ (with coordinate-wise multiplication), where $m = \dim A$.

\end{itemize}

Using these observations it easily follows that the proof of  Lemma 1.1 of \cite{FK5} holds
when semisimple is replaced by
semiprime and we have the following lemma.

\begin{lemma}
Let $B$ be a unital commutative semiprime Banach
algebra with connected character space, and let $T$ be a bounded unital
quasicompact endomorphism of $B$.
Then $1$ is an eigenvalue of $T$ with multiplicity $1$ and eigenspace
$\C \cdot 1$, and
$\sigma(T)$ (the spectrum of $T$) is contained in
$\{\lambda:|\lambda| < 1\} \cup \{1\}$.\QED
\end{lemma}

Armed with this lemma, the proof of the convergence of the operators
$T^n$ in Theorem 1.2 of
\cite{FK5} is equally valid for semiprime algebras, and we obtain
the corresponding result for semiprime algebras.

\begin{theorem}
\label{spqc}
Let $B$ be a unital commutative semiprime Banach
algebra with connected character space, let $T$ be a bounded, unital
endomorphism of $B$, and let $\phi$ be the associated selfmap of $\Phi_B$.
Then $T$ is quasicompact if and only if
the operators $T^n$ converge in operator norm to a rank-one unital
endomorphism of $B$; in this case $\phi$ has a unique fixed point $x_0 \in \Phi_A$, and
the rank-one endomorphism above must be the endomorphism $b \mapsto \hat b(x_0) 1$.\QED
\end{theorem}

We immediately obtain the following useful corollary.
\begin{corollary}
\label{useful-corollary}
Let $B$ be a unital commutative semiprime Banach
algebra with connected character space, let $T$ be a bounded unital
endomorphism of $B$, and let $\phi$ be the associated selfmap of $\Phi_B$.
\begin{enumerate}
\item[(i)]
If $\phi$ has no fixed points in $\Phi_B$, then $T$ is not quasicompact.
\item[(ii)]
Otherwise, let $x_0 \in \Phi_B$ be
a fixed point of $\phi$. Then $T$ is quasicompact if and only if
the operators $T^n$ converge in operator norm to the rank-one unital
endomorphism of $B$ defined by $b \mapsto \hat b(x_0) 1$.\QED
\end{enumerate}
\end{corollary}

Note that, once we have found a fixed point $x_0$ of $\phi$, we can apply this corollary without having to check whether this fixed point is unique. However, if we \emph{do} know that $\phi$ has more than one fixed point, then Theorem \ref{spqc} tells us immediately that $T$ is not quasicompact.
\sskip
Let $B$ be a commutative Banach algebra without identity. Then (by the Shilov Idempotent Theorem again) $B$ has no non-zero idempotent elements if and only if $\Phi_{B^\#}$ is connected. One trivial special case of this is, of course, when $B$ is radical. Note that it is possible for $\Phi_{B^\#}$ to be connected when $\Phi_B$ is disconnected, and vice-versa.

\begin{corollary}
\label{radical}
Let $B$ be a commutative semiprime Banach algebra which has no non-zero idempotent elements, and let $T$
be a bounded endomorphism of $B$. Then $T$ is quasicompact if and
only if $T^n \rightarrow 0$ in operator norm.
\end{corollary}

\begin{proof}
Clearly, if $T^n \rightarrow 0$ in operator norm, then $T$ is quasicompact.

Conversely, suppose that $T$ is quasicompact.
Obviously $B$ has no identity element. We may extend $T$ to a bounded unital endomorphism $T^\#$ of the commutative unital semiprime Banach algebra $B^\#$, and then $T^\#$ is also quasicompact. By Theorem \ref{spqc}, the powers of $T^\#$ converge in operator norm to a rank-one unital endomorphism of $B^\#$. It follows that $T^n \rightarrow 0$ in operator norm.
\end{proof}

In particular, for radical semiprime Banach algebras we have the following corollary,
which will be needed in Section 4.

\begin{corollary}
\label{radical}
Let $R$ be a commutative radical semiprime Banach algebra, and let $T$
be a bounded endomorphism of $R$. Then $T$ is quasicompact if and
only if $T^n \rightarrow 0$ in the operator norm.\QED
\end{corollary}

\section{Extension to more general semiprime Banach algebras}

    We now wish to generalize these results to the setting where the
algebra is semiprime and the character space need not be connected.
A further examination of the proof of Theorem 1.2 of \cite{FK5} reveals immediately
that the following more general result holds.

\begin{lemma}
\label{ops-lemma}
Let $B$ be a unital commutative semiprime Banach
algebra, and let $T$ be a bounded unital
quasicompact endomorphism of $B$.
Suppose that \[
\sigma(T) \subseteq \{\lambda \in \C:|\lambda| < 1\} \cup \{1\}\]
and that
the eigenvalue $1$ of $T$ has multiplicity $1$.
Then the operators $T^n$ converge in operator norm to a rank-one unital
endomorphism $S$ of $B$.\QED
\end{lemma}

The method we use to obtain results when the character space is
disconnected is based on a standard technique involving orthogonal
idempotents.

\begin{theorem}
Let $B$ be a unital commutative semiprime Banach
algebra, and let $T$ be a bounded unital
quasicompact endomorphism of $B$.
Then there exists an $n \in \N$ such that
\begin{equation}
\label{condition-star}
\sigma(T^n) \subseteq \{\lambda\in \C:|\lambda| < 1\} \cup
\{1\}\,.
\end{equation}
For such $n$, the unital quasicompact endomorphism $T^n$ of $B$ has
the following properties.
\begin{enumerate}
\item[(i)]
The eigenspace of $T^n$ corresponding to eigenvalue $1$ is a
finite-dimensional, unital subalgebra of $B$ isomorphic to $\C^m$ for some
$m \in \N$, and hence spanned by $m$ orthogonal idempotents, say $e_1, e_2,
\dots,e_m$.
\item[(ii)]
Set $B_i=e_i B$ $(1 \leq i \leq m)$. Then (under an equivalent norm) each $B_i$ is a commutative,
unital semiprime Banach algebra, with identity $e_i$, and
\[
B=\bigoplus_{i=1}^m B_i\,.
\]
\item[(iii)]
For $1\leq i\leq m $, $T^n|_{B_i}$ is a unital quasicompact
endomorphism of $B_i$, and $T^n|_{B_i}$ satisfies the conditions of Lemma \ref{ops-lemma}. The operators $(T^{kn}|_{B_i})_{k=1}^\infty$ converge in
operator norm to a rank-$1$ unital endomorphism of $B_i$, say $S_i$.
\item[(iv)]
The operators $(T^{kn})_{k=1}^\infty$ converge in operator norm to the
rank-$m$ endomorphism $S$ of $B$ given by
\[S(b)=\sum_{i=1}^m S_i(b e_i) ~~~(b \in B).\]
\end{enumerate}
\end{theorem}
\begin{proof}
As in the proof of Lemma 1.1 of \cite{FK5},  the existence of an $n$ satisfying (\ref{condition-star})
is an easy consequence of the following pair of facts: the set of
eigenvalues of $T$ is closed under taking powers and the spectrum of $T$ has no
limit point on the unit circle.

Now suppose that we have fixed such an $n$ satisfying (\ref{condition-star}). Then
(i) follows immediately from the fact that $\ker(I-T^n)$ is a
finite-dimensional, commutative semiprime algebra.
Now (ii) is a standard construction.
For (iii), it is clear that $T^n|_{B_i}$ is a unital endomorphism of $B_i$, and the multiplicity of the eigenvalue $1$ of this endomorphism is $1$ by construction. The quasicompactness of $T^n|_{B_i}$ is standard. Then, since every eigenvalue of $T^n|_{B_i}$ is also in $\sigma(T^n)$, it follows that $\sigma(T^n|_{B_i}) \subseteq \{\lambda\in \C:|\lambda| < 1\} \cup
\{1\}\,$. Thus $T^n|_{B_i}$ satisfies the conditions of Lemma \ref{ops-lemma}. The rest of (iii) now follows by applying Lemma \ref{ops-lemma} to $T^n|_{B_i}$. Finally,
(iv) follows immediately from (i), (ii) and (iii).
\end{proof}


Theorem \ref{general-spqc} is now an immediate corollary, since one implication is part of
the result above, while the converse is trivial.

\section{Radical Banach algebras of power series}

In this section we look briefly at radical Banach algebras of power
series.

We recall the following terminology and notation from \cite{Grab}. The algebra of complex formal power series in one variable is denoted by $\C[[z]]$. The coordinate projections on $\C[[z]]$ are $(\pi_n)_{n=0}^\infty$. Let $B$ be a subalgebra of $\C[[z]]$ with $z\in B$ and such that $B\subseteq \ker \pi_0$ (i.e., all elements of $B$ have constant coefficient $0$). Then $B$ is a \emph{generalized Banach algebra of power series} if it is a Banach algebra under some norm for which all of the functionals $\pi_n|_B$ are continuous.\footnote{In fact, surprising recent results from \cite{DPR} show that the continuity of the functionals $\pi_n|_B$ in this setting is automatic, while the corresponding statement for formal power series in two variables is false.}
In this case, for each $n \in \N$, we denote by $\|\pi_n\|$ the operator norm of the continuous linear functional $\pi_n|_B$. If $B$ is a generalized Banach algebra of power series such that the polynomials are dense in $B$, then $B$ is a \emph{Banach algebra of power series}. Since $\C[[z]]$ is an integral domain, these algebras of power series are certainly semiprime.

The reader should note that there are variations in the terminology and notation used in the literature. In \cite[Section 4.6]{dales}, for example, the algebras are allowed to be unital, and generalized Banach algebras of power series are called simply  Banach algebras of power series (with no requirement that the polynomials be dense).

Let $B$ be a generalized Banach algebra of power series. For each non-negative integer $j$, $S_{-j}(B)$ is the set of those formal power series $f$ with zero constant term for which $fz^j$ belongs to $B$. In fact $S_{-j}(B)$ is a Banach space when we define the norm of $f\in S_{-j}(B)$ to be the norm of $fz^j$ in $B$.

Let $B$ be a Banach algebra of power series. Then
every non-zero endomorphism of $B$ has the form
$f \rightarrow f \circ g$ (formal composition of power series) for some $g \in B$ \cite[p.7]{Grab}.
For those $g\in B$ which give rise to an endomorphism of $B$ in this way, we denote the corresponding endomorphism by $T_g$. In this case, we have $g = T_g z.$

A result of Loy \cite[Theorem 5.2.20]{dales} shows that endomorphisms of Banach
algebras of power series are automatically continuous. 
(See also \cite{DPR} for some striking recent developments concerning 
Fr\'{e}chet algebras of power series.)

It was previously shown that for
a wide class of radical Banach algebras of power series, every 
endomorphism is either an automorphism or compact
\cite[Theorem (2.6)]{Grab}. In such  cases
every quasicompact endomorphism is (trivially) compact. In particular, for many radical weights $\omega$, the Banach algebras
$\ell^p(\omega)$ are examples of radical
semiprime commutative Banach algebras for which every quasicompact
endomorphism is compact.

\begin{lemma} Let $B$ be a radical Banach algebra of
power series, and let $T$ be a quasicompact endomorphism of $B$. Set $g=Tz$ (so that $T=T_g$). Then $|\pi_1(g)| < 1$.
\end{lemma}

\begin{proof}
Since the endomorphism $T=T_g$ is quasicompact, by Corollary \ref{radical},
$T_g^n \rightarrow 0$ in norm.
In particular, $T_g^nz \rightarrow 0.$ But $\pi_1(T_g^n z)=\pi_1(g)^n$, and so we must have $|\pi_1(g)| < 1.$
\end{proof}

The following proposition is \cite[Theorem 5.7]{Grab}.
\begin{proposition}
Suppose that $B$ and $S_{-1}(B)$
are both radical
generalized Banach algebras of power series, and that
$R:=\limsup(\|\pi_n\|\|z^n\|)^{1/n}$ is finite. Let $g \in B$ with
$|R\pi_1(g)| < 1$.  Then $T_g$ is a compact endomorphism of $B$.\QED
\end{proposition}

Combining the previous two results, we have the following.

\begin{corollary}
\label{radical-power-series}
Suppose that $S_{-1}(B)$ and $B$
are radical
Banach algebras of power series, and
that $\limsup(\|\pi_n\|\|z^n\|)^{1/n} = 1$.
Then every quasicompact endomorphism of $B$ is compact.\QED
\end{corollary}

Let $A$ be $A(\D)$ or $H^p(\OD)$ for some $p \in [1,\infty)$.
Using the definition of convolution multiplication on $A$ from Section 1, it was
shown in \cite[Section 13]{Grab}, that $(A,\ast)$ is a
commutative radical semiprime Banach algebra
which can be identified with a radical Banach algebra of power series $B$
satisfying the hypotheses of Corollary \ref{radical-power-series} (see, in particular, \cite[Theorem (13.10)]{Grab}). Thus, for $B$, and hence also for $(A,\ast)$, every
quasicompact endomorphism is compact.

\section{Two semisimple examples}

We have just seen several examples of commutative radical semiprime
Banach algebras where every
quasicompact endomorphism is compact. In this section we give two examples of
commutative semisimple Banach algebras where this holds. This is in
contrast to the commutative semisimple Banach algebra $C^1[0,1]$ where there exist a
quasicompact endomorphism which is not Riesz, a Riesz endomorphism
which is not power compact and a power compact endomorphism which is
not compact.

\begin{example}
A theorem of Beurling and Helson \cite[Theorem 4.5 and exercise 4.12]{Katznelson} tells us that every non-zero
 endomorphism of the group algebra
$L^1(\R)$ is an automorphism. Thus, for this commutative semisimple Banach
algebra, there are no non-zero quasicompact endomorphisms at all.
\end{example}

For the next example, the proof is based on our results concerning the powers of quasicompact endomorphisms.

\begin{example}
Let $A$ be the Banach algebra $Ea[-1,1]$ described in
\cite{BD} and \cite{KW},  and defined as follows.
Let $M(\C)$ denote the set of finite regular Borel
    measures on $\C$ and $M^\omega(\C)$ the set of measures $\mu \in
    M(\C)$ for which $\int_\C e^{|Re \lambda|}\,\rd|\mu|(\lambda) < \infty.$
    For each $\mu \in M^\omega(\C)$, we may define a continuous function $f_\mu:[-1,1]\rightarrow \C$ by
    $f_\mu(x)=\int_\C e^{x\lambda}\,\rd\mu(\lambda)$ ($x \in [-1,1]$).
    Then
    \[
    Ea[-1,1] = \{f_\mu: \mu \in M^\omega(\C)\}\,.
    \]
    With norm defined by

    \[
    \|f\|_A=\inf\left\{\int_\C e^{|Re \lambda|}\,\rd|\mu|(\lambda): \mu \in M^\omega(\C) \text{ with }
    f_\mu=f\,\right\}\,,
    \]
      $A=Ea[-1,1]$  is a regular
    commutative semisimple Banach algebra
    \cite{BD,Sin}.
    Further $\Phi_A$ is $[-1,1]$. This
    algebra is called the \textit{extremal algebra} for $[-1,1]$, a name
    derived from a property it possesses relative to the study of
    numerical ranges of elements in complex unital Banach
    algebras. The Banach algebra $A$ is generated by the Hermitian element $u$,
    where $u(x)=x$ for $x \in [-1,1]$. Also $\|e^{itu}\|_A=1$
    for all real $t$ \cite{BD,Sin}.

Let $T$ be an endomorphism of $A$, and let $\phi$ be the associated selfmap of $[-1,1]$.
In \cite{KW} it was shown that $\phi$ must have the form $x \mapsto \alpha x+\beta$, where $\alpha$ and $\beta$ are real numbers with $|\alpha| + |\beta| \leq 1$. If $\beta = 0$ and $|\alpha| = 1$, then $T$ is an automorphism, while if
$\alpha=0$, then the endomorphism $T$ has rank one, and so $T$ is compact.
Also it was shown in \cite{KW} that
the rank-one endomorphisms are the only nonzero compact endomorphisms of $A$.
We claim that every quasicompact endomorphism of
$A$ is compact.

To see this, let $T$ be a quasicompact endomorphism of
$A$, with associated selfmap $\phi$. For some $\alpha$ and $\beta$ as above, we have
$\phi(x) = \alpha x+\beta$. Since $T$ is not an automorphism,  we have $\alpha \neq 1$.
Set $x_0=\frac{\beta}{1-\alpha}$, so that $x_0$ is the fixed point of $\phi$.
Let
$S$ be the rank-one endomorphism $f \mapsto f(x_0)1$. Since $T$ is quasicompact, Corollary \ref{useful-corollary} implies
that $\|T^n-S\| \rightarrow 0.$
We shall show that $\alpha = 0$, and hence that $x_0=\beta$ and $T=S$.

Suppose, towards a contradiction, that $\alpha \neq 0$.
For each $n \in \N$, define $f_n \in A$ by
\[
f_n(x)=\left(\frac{1 + \exp(i(\alpha^{-n}(x -x_0)))}{2}\right)^n\,.
\]
As mentioned above, for each real number $t$, the function $x \mapsto e^{i t x}$ has norm $1$ in $A$. Thus we have
\[
1=\|f_n\|_\infty \leq \|f_n\|_{A} \leq 1\,,\]
and so $\|f_n\|_{A} = 1.$

It is routine to show that for each positive integer $n$,
$T^nf(x)=f(\alpha^n(x-x_0)+x_0)$ and so
\[
T^n f_n(x)=\left(\frac{1+e^{i(x-x_0)}}{2}\right)^n\,.
\]
We also note that
$Sf_n(x)=f_n(x_0)=1$.
Now
\[
\|f_n\|_A\|T^n-S\| \geq \|(T^n-S)f_n\|_A \geq |(T^n-S)f_n(x)|
\]
for all $x$ in $[-1,1]$.
Evaluate at some $x_1 \neq x_0.$ Then
\[
\|T^n - S\|=\|f_n\|_A\|T^n-S\| \geq \|(T^n-S)f_n\| \geq |(T^n-S)f_n(x_1)|\,.
\]
This
implies that
\[
\|T^n-S\| \geq \left| \left(\frac{1+e^{i(x_1-x_0)}}{2}\right)^n - 1\right| \geq 1/2
\]
for large $n$. Therefore $\|T^n-S\|$ does not converge to $0$,
and so
$T$ is not quasicompact according to Corollary \ref{useful-corollary} (or Proposition \ref{ssqc}).
This contradiction shows that $\alpha=0$, and so $T=S$.

Therefore for this algebra, every
quasicompact
endomorphism is compact.
\end{example}

\section*{Acknowledgements}
We would like to thank the referee for useful comments.
We would also like to thank Garth Dales and Sandy Grabiner for some very helpful discussions.
Our discussions with Sandy Grabiner took place at the 19$^\mathrm{th}$
International Conference on Banach Algebras held at B\c{e}dlewo, July 14--24,
2009. The support for the meeting by the Polish Academy of Sciences, the
European Science Foundation under the ESF-EMS-ERCOM partnership, and the
Faculty of Mathematics and Computer Science of the Adam Mickiewicz University
at Pozna\'n is gratefully acknowledged.
In addition to the support for both authors from the Conference's funds, the first author's attendance at this meeting was further supported by a grant from the Royal Society.

\end{document}